\documentclass[a4paper,12pt,leqno]{amsart}
\makeatletter
\usepackage{amsfonts,delarray,amssymb,amsmath,amsthm,a4,a4wide}

\usepackage{color}

\title[Curvature behavior at the first singular time of the Ricci flow]{Remarks on curvature behavior at the first singular time of the Ricci flow}

\author{Nam Q.  Le}

\address{Department of
Mathematics, Columbia University, New York,
 USA}
\email{namle@math.columbia.edu}

\author{Natasa Sesum$^{*}$}
\address{Department of Mathematics, 
University of Pennsylvania, Philadelphia, PA,
USA}
\email{natasas@math.upenn.edu}

\thanks{$*:$ Partially supported
by NSF grant 0905749}

\newcommand{\review}[2][\right]{\relax
\ifx#1\right\relax \left.\fi#2#1\rvert}
\let\abs=\envert
 \newtheorem{definition}{Definition}[section]
\newtheorem{theorem}{Theorem}[section]
\newtheorem{propo}{Proposition}[section]
\newtheorem{remark}{Remark}[section]
\newtheorem{cor}{Corollary}[section]
\newtheorem{lemma}{Lemma}[section]

\newcommand{\bef}{\begin{flushright}}
\newcommand{\eef}{\end{flushright}}

\newcommand{\eval}[2][\right]{\relax
\ifx#1\right\relax \left.\fi#2#1\rvert}

\let\abs=\envert

\numberwithin{equation}{section}
\let\norm=\enVert
\newcommand\e{\varepsilon}
\newcommand{\h}{\hspace*{.24in}}

\newcommand{\vol}{\mathrm{vol}}
\def\h{\hspace*{.24in}}

\def\beq{\begin{eqnarray*}}
\def\eeq{\end{eqnarray*}}
\def\r{\rho}
\def\ra{\rightarrow}

\newcommand{\ric}{\mathrm{Ric}}

\newcommand{\rem}{\mathrm{Rm}}

\def\RR{\mbox{$I\hspace{-.06in}R$}}

\def\m{\mu}

\begin{document}
\maketitle
\author
\pagenumbering{arabic}

\begin{abstract}
In this paper, we study curvature behavior at the first singular time of solution to the Ricci flow on a smooth, compact n-dimensional 
Riemannian manifold $M$, $\frac{\partial}{\partial t}g_{ij} = -2R_{ij}$ for $t\in [0,T)$.
If the flow has uniformly bounded scalar curvature and develops Type I 
singularities at $T$,  using Perelman's $\mathcal{W}$-functional, we show that suitable blow-ups of our evolving metrics converge in the pointed
Cheeger-Gromov sense to a
Gaussian shrinker. If the flow has uniformly bounded scalar curvature and develops Type II 
singularities at $T$, we show that suitable scalings of the potential functions in Perelman's entropy functional converge to a
positive constant on a complete, Ricci flat manifold. 
We also show that if the scalar 
curvature is uniformly bounded along the flow in certain  integral sense then the flow either develops a type II singularity at $T$ or it can be smoothly 
extended past time $T$.
\end{abstract}

\section{Introduction}
\subsection{The Ricci flow and previous results}
Let $M$ be a smooth, compact n-dimensional Riemannian manifold without boundary and equipped with a smooth Riemannian metric $g_{0}$ ($n\ge 3$). Let $g(t)$ ( $0\leq t <T$) be a one-parameter family of 
metrics on $M$.
The Ricci flow equation on $M$ with initial metric $g_{0}$
\begin{eqnarray}
\label{eq-RF}
\frac{\partial}{\partial t} g(t) &=& -2\ric (g(t)), \\
g(0) &=& g_0. \nonumber
\end{eqnarray}
has been introduced by Hamilton in his seminal paper \cite{Ha1}. It is a weakly parabolic system of equations whose short time existence was proved by Hamilton using the Nash-Moser
implicit function theorem in 
the same paper and after that simplified by DeTurck \cite{DT}. The goal in the analysis of (\ref{eq-RF}) is to understand the long time behavior of the flow, possible singularity 
formation or convergence of the flow in the cases when we do have a long time existence. In general, the behavior of the flow can serve to give us more
insights about the topology of the underlying manifold. One of the great successes is the resolution of the Poincar\'{e} Conjecture by Perelman. 
In order to discuss those things we have to understand what happens at the  singular time and also what the optimal conditions for having a smooth solution are.  

In \cite{Ha} Hamilton showed that if the norm of  Riemannian curvature $|\rem|(g(t))$ stays uniformly bounded in 
time, for all $t\in [0,T)$ with $T < \infty$, then we can extend the flow (\ref{eq-RF}) smoothly past time $T$. In other words, either the 
flow exists forever or the norm of Riemannian curvature blows up in finite time. 
{This result has been extended in \cite{W} and \cite{YeII},  assuming certain integral bounds on the Riemannian curvature. Namely, if
$\int_0^T\int_M |Rm|^{\alpha}\, dvol_{g(t)}\, dt \le C$, for some $\alpha \ge \frac{n+2}{2}$ then the flow can be extended smoothly past time $T$. Throughout the paper,
we denote $dvol_{g}$ the Riemannian volume density on $(M, g)$. On the other hand,
in \cite{Se} Hamilton's  extension result has been improved} by the second author and it was shown that if the norm of  Ricci curvature is uniformly bounded over a 
finite time interval $[0,T)$, then we can extend the flow smoothly past time $T$. In \cite{W} this has been improved even further. That is, if Ricci curvature is 
uniformly bounded from below and if the space-time  integral of the scalar curvature is bounded, say $\int_0^T\int_M |R|^{\alpha}\, dvol_{g(t)}\, dt \le C$ for $\alpha \ge \frac{n+2}{2}$, where $R$ is the 
scalar curvature, then Wang showed that we can extend the flow smoothly past time $T$. {The requirement on  Ricci curvature in \cite{W} is rather restrictive. Ricci flow 
does not in general preserve nonnegative Ricci curvature in dimensions $n\geq 4$.
See Knopf \cite{Knopf} for non-compact examples starting in dimension $n =4$ and B\"{o}hm and Wilking \cite{BW} for compact examples starting in dimension $n=12$. Without
assuming the boundedness from below of Ricci curvature, Ma and Cheng \cite{MC} proved that the norm of  Riemannian curvature can be controlled provided that
one has the integral bounds on the scalar curvature $R$ and the Weyl tensor $W$ from the orthogonal decomposition of the Riemannian curvature tensor. Their bounds are of
the form $\int_0^T\int_M (|R|^{\alpha}+ \abs{W}^{\alpha})\, dvol_{g(t)}\, dt \le C$, for some $\alpha \ge \frac{n+2}{2}$.
}In \cite{Zh} it has been proved that the scalar curvature controls the K\"ahler Ricci flow $\frac{\partial}{\partial t} g_{i\bar{j}} = -R_{i\bar{j}} - g_{i\bar{j}}$ 
starting from any K\"ahler metric $g_0$. 
\subsection{Main results} The above results, in particular that in \cite{Zh},  support the belief that 
the scalar curvature should control the Ricci flow in the Riemannian setting as well.  In \cite{EMT}, Enders, M\"{u}ller and Topping justified this belief
for Type I Ricci flow, that is, they proved the following theorem.

\begin{theorem}[Enders, M\"uller, Topping]
\label{thm-main}
Let $M$ be a smooth, compact n-dimensional Riemannian manifold equipped with a smooth Riemannian metric $g_{0}$ and $g(\cdot,t)$ be a solution to the  Type I Ricci 
flow equation (\ref{eq-RF}) on $M$. Assume there is a constant $C$ so that $\sup_M |R(\cdot,t)| \le C$, for all $t\in [0,T)$ and $T < \infty$. Then  we can 
extend the flow past time $T$.
\end{theorem}
Their proof was based on a blow-up argument using Perelman's reduced distance and pseudolocality theorem. \\
\h Assume the flow (\ref{eq-RF}) develops a singularity at $T < \infty$.
Throughout the paper, we use the following
\begin{definition}
 We say that (\ref{eq-RF}) has a Type I singularity at $T$  if there exists a constant $C>0$ such that for all $t\in [0, T)$
\begin{equation}
\max_{M} \abs{Rm(\cdot, t)}\cdot (T-t) \leq C.
 \label{TypeI-flow}
 \end{equation}
Otherwise we say the flow develops Type II singularity at $T$. Moreover, the flow that satisfies (\ref{TypeI-flow}) will be referred to as to the {\bf Type I Ricci flow.}
\end{definition}

\h In this paper, we also use a blow-up argument to study
curvature behavior at the first singular time of the Ricci flow. We deal with both Type I and II singularities. Assume that the scalar curvature is uniformly
bounded along the flow. If the flow develops Type I 
singularities at some finite time $T$ then by using Perelman's entropy functional $\mathcal{W}$, we show that suitable blow-ups of our evolving metrics converge in the pointed
Cheeger-Gromov sense to a
Gaussian shrinker.
\begin{theorem}
\label{prop-typeI}
Let $M$ be a smooth, compact $n$-dimensional Riemannian manifold ($n\geq 3$) and $g(\cdot,t)$ be a solution to the Ricci flow equation (\ref{eq-RF}) on $M$. Assume 
there is a constant $C$ so that $\sup_M |R(\cdot,t)| \le C$, for all $t\in [0,T)$ and $T < \infty$. 
Assume that at $T$ we have a type I singularity and the norm of the curvature operator blows up.
Then by suitable rescalings of our metrics, we get a Gaussian shrinker in the limit.
\end{theorem}

A simple consequence of the proof of previous theorem is following result, which is the same to the one proved by Naber in \cite{Na}. The  difference is 
that instead of the reduced distance techniques used by Naber, we use Perelman's monotone functional $\mathcal{W}$.

\begin{cor}
Let $M$ be a smooth, compact $n$-dimensional Riemannian manifold ($n\geq 3$) and $g(\cdot,t)$ be a solution to the Ricci flow equation (\ref{eq-RF}) on $M$. If the flow has type I singularity at $T$, then a suitable rescaling of the solution converges to a gradient shrinking Ricci soliton. 
\end{cor}

We also have the following consequence.

\begin{cor}
\label{cor-vol}
Let $M$ be a smooth, compact $n$-dimensional Riemannian manifold ($n\geq 3$) and $g(\cdot,t)$ be a Type I solution to the Ricci flow equation (\ref{eq-RF}) on $M$. There exists a $\delta > 0$ so that if $|R|(g(\cdot,t)) \le C$ for all $t\in [0,T)$, then $\vol_{g(t)}(M) \ge \delta$, for all $t \in [0,T)$.
\end{cor}

In \cite{Na} it has been proved that in the case of type I singularity, a suitable rescaling of the flow converges to 
gradient shrinking Ricci soliton. In \cite{EMT}, it has been recently showed that the limiting soliton represents a singularity model, that is, it is nonflat. The open 
question is whether using Perelman's $\mathcal{W}$-functional, one can produce in the limit a singularity model ({\it nonflat} gradient shrinking Ricci solitons). We prove some interesting estimates on the minimizers of Perelman's $\mathcal{W}$-functional which can be of independent interest.
 
On the other hand, if the flow develops Type II 
singularities at some finite time $T$, then we show that suitable scalings of the potential functions in Perelman's entropy functional converge to a
positive constant on a complete, Ricci flat manifold which is the pointed Cheeger-Gromov limit of a suitably chosen sequence of 
blow-ups of our original evolving metrics. 
\begin{theorem}
Let $M$ be a smooth, compact $n$-dimensional Riemannian manifold ($n\geq 3$) and $g(\cdot,t)$ be a solution to the Ricci flow equation (\ref{eq-RF}) on $M$. Assume 
there is a constant $C$ so that $\sup_M |R(\cdot,t)| \le C$, for all $t\in [0,T)$ and $T < \infty$. 
Assume that at $T$ we have a type II singularity and the norm of the curvature operator blows up. Let $\phi_{i}$ be as in the proof of Theorem \ref{prop-typeI} (see
, e.g, (\ref{eq-min})).
Then by suitable rescalings of our metrics and $\phi_{i}$, we get as a limit of $\phi_{i}$ a positive constant on a complete, Ricci flat manifold.
\label{lem-typeII}
\end{theorem}

We believe that previous theorem may play a role in proving the nonexistence of type II singularities if the scalar curvature is uniformly bounded along the flow. We are still investigating that.

For a precise definition of $\phi_{i}$, see Section \ref{sec-main}.

{
There has been a striking analogy between the Ricci flow and the mean curvature flow for decades now. About the same time when Hamilton proved that 
the norm of the Riemannian curvature under the Ricci flow must blow up at a finite singular time, Huisken \cite{Hu} showed that the norm of  the second fundamental form of an 
evolving hypersurface under the mean curvature flow must blow up at a finite singular time.  In \cite{LS1} the authors showed that the analogue of Wang's result holds for the mean curvature 
flow as well, namely if the second fundamental form of an evolving hypersurface is uniformly bounded from below and if the  mean curvature is bounded in certain integral sense, then we can 
smoothly extend the flow. In the follow-up paper \cite{LS2} the authors show that if one only has the uniform bound on the mean curvature of the evolving 
hypersurface, then the flow either 
develops a type II singularity or can be smoothly extended. In the case the dimension of the evolving hypersurfaces is two they show that under some density assumptions 
one can smoothly extend the flow provided that the mean curvature is uniformly bounded. {Finally, we note that, in contrast to  
the lower bound on the scalar curvature (\ref{lower-scar}), at the first
singular time of the mean curvature flow, the mean curvature can either tend to $\infty$ (as in the case of a round sphere) or $-\infty$  as
in some examples of Type II singularities \cite{AV}.}

If we replace the pointwise scalar curvature bound in Theorem \ref{thm-main} with an integral bound we can prove the following theorems.

\begin{theorem}
If $g(\cdot,t)$ solves (\ref{eq-RF})  and if
\begin{equation}
\int_{M} \abs{R}^{\alpha} (t) dvol_{g(t)} \leq C_{\alpha} 
\label{int-alpha}
\end{equation}
for all $t\in [0, T)$ where $\alpha>n/2$ and $T < \infty$, then either the flow develops a type II singularity at $T$ or the flow can be smoothly extended past time $T$.
 \label{int-bound}
\end{theorem}

\begin{remark}
\label{rem1}
 The condition on $\alpha$ in Theorem \ref{int-bound} is optimal. Let $(S^{n}, g_{0})$ be the space form of constant sectional curvature 1. The Ricci flow on $M=S^{n}$
with initial metric $g_{0}$ has the solution $g(t) = (1- 2(n-1)t) g_{0}$. Therefore $T=\frac{1}{2(n-1)}$ is the maximal existence time. We can rewrite
$g(t) = 2(n-1) (T-t) g_{0}$ and compute
\begin{eqnarray*}
 \int_{M} \abs{R}^{\alpha}(t) dvol_{g(t)} = vol_{g(t)}(M) (\frac{n}{2(T-t)})^{\alpha} &=& vol_{g(0)}(M) \left(2(n-1)(T-t)\right)^{n/2}(\frac{n}{2(T-t)})^{\alpha}\\
&=& vol_{g(0)} (M) 2^{n/2-\alpha} (n-1)^{n/2}n^{\alpha} \frac{1}{(T-t)^{\alpha-n/2}}.
\end{eqnarray*}
Hence $\int_{M} \abs{R}^{\alpha} (t) dvol_{g(t)}$ tends to $\infty$ as $t\rightarrow T$ if and only if $\alpha>n/2$.
\end{remark}

\begin{theorem}
If $g(\cdot,t)$ is as above, then if we have the following space-time integral bound,
\begin{equation}
\int_{0}^{T}\int_{M} \abs{R}^{\alpha} (t) dvol_{g(t)} dt\leq C_{\alpha} 
\label{int-st}
\end{equation}
for $\alpha\geq \frac{n +2}{2}$, then the flow either develops a type II singularity at $T$  or can be smoothly extended past time $T$.
 \label{int-bound1}
\end{theorem}

\begin{remark}
The condition on $\alpha$ in Theorem \ref{int-bound1} is optimal . As in Remark \ref{rem1} consider the Ricci flow on the round sphere. Following 
the computation in Remark \ref{rem1} we get
$$\int_0^T\int_M |R|^{\alpha}\, dvol_{g(t)} dt = vol_{g(0)}(M)2^{n/2-\alpha} (n-1)^{n/2}n^{\alpha} \int_0^T\frac{1}{(T-t)^{\alpha-n/2}} \, dt,$$
and therefore the integral is  $\infty$ if and only if $\alpha \ge \frac{n+2}{2}$.
\end{remark}
\h For the mean curvature flow, similar results to Theorem \ref{int-bound1} have been obtained by the authors \cite{LS2}.

The rest of the paper is organized as follows. In Section \ref{preli} we will give some necessary preliminaries. Section \ref{sec-main} is 
devoted to the statements and proofs of Theorems \ref{prop-typeI}, \ref{lem-typeII} and Corollary \ref{cor-vol}. In 
section \ref{sec-cor} we prove Theorems \ref{int-bound} and \ref{int-bound1}.
\vglue 0.1cm
\noindent

{\bf Acknowledgements}: The authors would like to thank John Lott for helpful conversations during the preparation of this paper.
\section{Preliminaries}
\label{preli}
\h In this section, we recall basic evolution equations during the Ricci flow and the definition of  singularity formation. Then we recall Perelman's entropy functional $\mathcal{W}$ and prove one of its properties concerning the $\mu$-energy, Lemma \ref{lem-mu}. The 
nonpositivity of the $\mu$-energy turns out to be 
very crucial for the proof of Theorem \ref{thm-main}.
\subsection{Evolution equations and singularity formation} Consider the Ricci flow equation (\ref{eq-RF}) on $[0, T)$. Then, the scalar curvature $R$ and the volume form $vol_{g(t)}$ evolve by the following equations
\begin{equation}
 \frac{\partial}{\partial t} R = \Delta R + 2 \abs{Ric}^2
\label{evol-scar}
\end{equation}
and
\begin{equation}
 \frac{\partial}{\partial t} vol_{g(t)} = -R vol_{g(t)}.
\label{eq-vol}
\end{equation}
Because $|\ric|^2 \ge \frac{R^2}{n}$, the maximum principle applied to (\ref{evol-scar}) yields
\begin{equation}
 R(g(t))\geq \frac{min_{M}R(g(0))}{1- \frac{2 min_{M}R(g(0)) t}{n}}.
\label{lower-scar}
\end{equation}
If $T<+\infty$ and the norm of the Riemannian curvature $\abs{Rm}(g(t))$
becomes unbounded as $t$ tends to $T$, we say the Ricci flow
develops singularities as $t$ tends to $T$ and $T$ is a singular
time. It is well-known that the Ricci flow generally develops
singularities. \\
\h If a solution $(M,g(t))$ to the Ricci flow develops singularities
at $T < +\infty$, then according to Hamilton
\cite{Ha}, we say that it develops a \mbox{\textbf{Type I
singularity}} if
$$ \ \sup_{t\in [0,T)}(T-t) \max_{M} |Rm(\cdot,t)| <+\infty,\
$$
and it develops a  \mbox{\textbf{Type II singularity}} if
$$ \ \sup_{t\in [0,T)}(T-t) \max_{M} |Rm(\cdot,t)| =+\infty.
$$
Clearly, the Ricci flow of a round sphere develops Type I singularity in finite time. The existence of type II singularities for the Ricci flow has been recently
established by Gu and Zhu \cite{GZ}, proving the degenerate neckpinch conjecture of Hamilton \cite{Ha}. \\
\h Finally, by the curvature gap estimate for Ricci flow solutions with finite time singularity (see, e.g., Lemma 8.7 in \cite{CLN}), we have
\begin{equation}
 \max_{x\in M} \abs{Rm(x,t)}\geq \frac{1}{8(T-t)}.
\label{curgap}
\end{equation}
\subsection{Perelman's entropy functional $\mathcal{W}$ and the $\mu$-energy}
In \cite{Pe} Perelman has introduced a very important functional, the entropy functional $\mathcal{W}$, for the study of the Ricci flow,
\begin{equation}
\label{eq-perelman}
\mathcal{W}(g,f,\tau) = (4\pi\tau)^{-n/2}\int_M [\tau(R + |\nabla f|^2) + f - n] e^{-f}\, dvol_{g},
\end{equation}
under the constraint $(4\pi\tau)^{-n/2}\int_M e^{-f}\, dvol_{g} = 1$. The functional $\mathcal{W}$ is invariant under the parabolic scaling of the Ricci flow and invariant
under diffeomorphism. Namely, for any positive number $\alpha$ and any diffeomorphism $\varphi$, we have 
$\mathcal{W}(\alpha\varphi^{\ast} g, \varphi^{\ast} f, \alpha \tau) = \mathcal{W}(g,f,\tau).$
Perelman showed that if $\dot{\tau} = -1$ and $f(\cdot,t)$ is a solution to the backwards heat equation 
\begin{equation}
\label{back}
\frac{\partial f}{\partial t}  = -\Delta f + |\nabla f|^2 - R + \frac{n}{2\tau},
\end{equation}
and $g(\cdot,t)$ solves the Ricci flow equation (\ref{eq-RF}) then
\begin{equation}
\label{eq-monotone}
\frac{d}{dt}\mathcal{W}(g(t),f(t),\tau) = (2\tau)\cdot (4\pi\tau)^{-n/2}\int_M |R_{ij} + \nabla_i\nabla_j f - \frac{g_{ij}}{2\tau}|^2 e^{-f}\, dvol_{g(t)} \ge 0.
\end{equation}
We see that $\mathcal{W}$ is constant on metrics $g$ with the property that
$$R_{ij} + \nabla_i\nabla_j f - \frac{g_{ij}}{2\tau} = 0,$$
for a smooth function $f$.
These metrics are called gradient shrinking Ricci solitons and appear often as singularity models, that is, limits of  blown up solutions around 
finite time singularities of the Ricci flow. \\
\h Let $g(t)$ be a solution to the Ricci flow (\ref{eq-RF}) on $(-\infty, T)$. We call a triple $(M, g(t), f(t))$ on $(-\infty, T)$ with smooth functions $
f: M\rightarrow \RR$ a {\bf gradient shrinking soliton in canonical form} if it satisfies
\begin{equation}
 \ric(g(t))  + \nabla^{g(t)} \nabla^{g(t)} f(t) - \frac{1}{2(T-t)} g(t) =0 ~~\text{and}~ \frac{\partial}{\partial t} f(t) = \abs{f(t)}^{2}_{g(t)}.
\end{equation}
\h Perelman also defines the $\mu$-energy
\begin{equation}
\mu(g,\tau) = \inf_{\{f\,|\,(4\pi\tau)^{-n/2}\int_M e^{-f}\, dvol_{g}= 1\}} \mathcal{W}(g,f,\tau),
\label{mu-energy}
\end{equation}
and shows that 
\begin{equation}
\label{mon-mu}
\frac{d}{dt}\mu(g(\cdot,t),\tau) \ge (2\tau)\cdot(4\pi\tau)^{-n/2}\int_M |R_{ij} + \nabla_i\nabla_j - \frac{g_{ij}}{2\tau}|^2 e^{-f}\, dvol_{g(t)} \ge 0,
\end{equation}
where $f(\cdot,t)$ is the minimizer for $\mathcal{W}(g(\cdot,t), f, \tau)$ with the constraint on $f$ as above.  Note that $\mu(g,\tau)$ corresponds to the best constant 
of a logarithmic Sobolev inequality.
Adjusting some of Perelman's arguments to our situation we get the following Lemma whose proof we will include below for reader's convenience.

\begin{lemma}[\it\bf Nonpositivity of the $\mu$-energy ]
\label{lem-mu}
If $g(t)$ is a solution to (\ref{eq-RF}) for all $t\in [0,T)$, then
$$\mu(g(t),T-t) \le 0, \,\,\, \mbox{for all} \,\,\, t\in[0,T).$$
\end{lemma}

\begin{proof}
We are assuming the Ricci flow exists for all $t\in [0,T)$. Fix $t \in [0,T)$. Define $\tilde{g}(s) = g(t+s)$, for $s\in [0, T-t)$. Pick any $\bar{\tau} < T-t$.
Let $\tau_0=\bar{\tau}-\e$ with $\e>0$ small. Pick $p\in M$. We use normal coordinates
about $p$ on $(M,\tilde{g}(\tau_0))$ to define
\begin{equation}
f_1 (x)=
\begin{cases}
&\frac{|x|^2}{4\e} \qquad \text{if}~d_{\tilde{g}(\tau_{0})}(x,x_0)<\r_0, \\
&\frac{\r_0^2}{4\e} \qquad \text{elsewhere}
\end{cases}
\end{equation}
where $\r_0>0$ is smaller than the injectivity radius. Note that $dvol_{\tilde{g}(\tau_{0})}(x) = 1+O(|x|^2)$ near $p$.
We compute
\begin{align*}
\int_M(4\pi\e)^{-n/2}e^{-f_1}dvol_{\tilde{g}(\tau_{0})}
&=\int_{|x|\leq\r_0}(4\pi\e)^{-n/2}e^{-|x|^2/4\e}(1+O(|x|^2))dx+O(\e^{-n/2}e^{-\r_0^2/4\e}) \\
&=\int_{|y|\leq\r_0/\sqrt{\e}}(4\pi)^{-n/2}e^{-|y|^2/4}(1+O(\e|y|^2))dy
  +O(\e^{-n/2}e^{-\r_0^2/4\e})
\end{align*}
The second term goes to zero as $\e\ra 0$ while the first term converges to
$$\int_{\mathbb{R}^n}(4\pi)^{-n/2}e^{-|y|^2/4}dy=1. $$
If we write the integral as $e^{C}$, then
$C\ra 0$ as $\e\ra 0$. And $f=f_1+C$ then satisfies the constraint $\int_M(4\pi\e)^{-n/2}e^{-f}dvol_{\tilde{g}(\tau_{0})} =1 $.

We solve the equation (\ref{back}) backwards with initial value $f$ at $\tau_0$. Then
\begin{align*}
&\quad \mathcal{W}(\tilde{g}(\tau_0),f(\tau_0),\bar{\tau}-\tau_0) \\
&=\int_{|x|\leq\r_0}[\e(\frac{|x|^2}{4\e^2}+R)+\frac{|x|^2}{4\e}+C-n]
(4\pi\e)^{-n/2}e^{-|x|^2/4\e-C}(1+O(|x|^2))dx \\ 
&\quad + \int_{M-B(p,\r_0)}(\frac{\r_o^2}{4\e}+\e R+C-n)
(4\pi\e)^{-n/2}e^{-r_0^2/4\e-C} \\
&=I+II,
\end{align*}
where
$I=e^{-C}\int_{|x|\leq\r_0}(\frac{|x|^2}{2\e}-n)(4\pi\e)^{-n/2}e^{-|x|^2/4\e}(1+O(|x|^2))dx$
and $II$ contains all the remaining terms. It is obvious that $II\ra 0$ as $\e\ra 0$ while
\begin{align*}
I&=e^{-C}\int_{|y|\leq\r_0/\sqrt{\e}}(\frac{|y|^2}{2}-n)(4\pi)^{-n/2}e^{-|y|^2/4}(1+O(\e|y|^2))dy \\
&\quad \ra\int_{\mathbb{R}^n}(\frac{|y|^2}{2}-n)(4\pi)^{-n/2}e^{-|y|^2/4}dy=0~\text{as}~ \e\rightarrow 0.
\end{align*}
Therefore $\mathcal{W}(\tilde{g}(\tau_0),f(\tau_0),\bar{\tau}-\tau_0)\ra 0$ as $\tau_0\ra\bar{\tau}$.
By the monotonicity of $\mu$ along the flow,
$\m(g(t),\bar{\tau}) = \mu(\tilde{g}(0), \bar{\tau}) \leq \mathcal{W}(\tilde{g}(0),f(0),\bar{\tau})\leq \mathcal{W}(\tilde{g}(\tau_0),f(\tau_0),\bar{\tau}-\tau_0).$
Let $\tau_0\ra \bar{\tau}$, we get $\m(g(t),\bar{\tau})\leq 0$. Since $\bar{\tau} < T - t$ is arbitrary we get
$$\mu(g(t), T - t) \le 0.$$
\end{proof}

\section{Uniform bound on scalar curvature}
\label{sec-main}
In this section, we prove Theorems \ref{prop-typeI}, \ref{lem-typeII} and Corollary \ref{cor-vol}.

\begin{proof}[Proof of Theorem \ref{prop-typeI}]
By our assumptions, there exists 
a sequence of times $t_i\to T$ so that
$Q_i := \max_{M\times[0,t_i]}|\rem|(x,t) \to \infty$ as $i\to\infty$. Assume that the maximum is achieved at $(p_i, t_{i})\in M\times [0, t_i]$. 
Define a rescaled sequence of solutions
\begin{equation}
\label{eq-rescaling}
g_i(t) = Q_i\cdot g(t_i + t/Q_i).
\end{equation}
We have that 
\begin{equation}
\label{eq-cur-bound}
|\rem(g_{i})| \le 1 \,\,\, \mbox{on} \,\,\, M\times[-t_i Q_i,0] \,\,\, \mbox{and} \,\,\, |\rem (g_{i})|(p_i,0) = 1.
\end{equation}
By Hamilton's compactness theorem \cite{Ha2} and Perelman's $\kappa$-noncollapsing theorem \cite{Pe} we can extract a pointed subsequence of 
solutions $(M, g_i(t), q_i)$, 
converging in the Cheeger-Gromov sense to a solution to (\ref{eq-RF}), which we denote by $(M_{\infty}, g_{\infty}(t), q_{\infty})$, for any sequence of points $q_i
\in M$. In particular, if we take that sequence of points to be exactly $\{p_i\}$, we can guarantee the limiting metric is nonflat.  The limiting metric has 
a sequence of nice properties that we discuss below. Since 
$$|R(g_i(t))| = \frac{|R(g(t_i+\frac{t}{Q_i}))|}{Q_i} \le \frac{C}{Q_i} \to 0,$$
our limiting solution $(M_{\infty}, g_{\infty}(t))$ is scalar flat, for each $t\in (-\infty, 0]$. Since it solves the Ricci 
flow equation (\ref{eq-RF}) and $R_{\infty}:= R(g_{\infty})$ evolves by
$$\frac{\partial}{\partial t} R_{\infty} = \Delta R_{\infty} + 2\abs{\ric(g_{\infty})}^2,$$
we have that $\ric(g_{\infty}) \equiv 0$, that is, the limiting metric is Ricci flat. We will get a Gaussian shrinker by using 
Perelman's  functional $\mu$ 
defined by (\ref{mu-energy}).
Recall that (see the computation in \cite{KL}) 
$$\frac{d}{dt}\mu(g(t),\tau) \ge 2{\tau}\cdot (4\pi\tau)^{-n/2}\int_M |\ric + \nabla\nabla f - \frac{g}{2\tau}|^2 e^{-f}\, dvol_{g(t)},$$
where $f(\cdot,t)$ is the minimizer realizing $\mu(g(t),\tau)$, and $\tau = T - t$.\\
\h {\it In this Theorem \ref{prop-typeI}, we take $s, v\in [-10, 0]$ with $s<v$}. Then, by (\ref{eq-cur-bound}), $g_{i}(s)$ and $g_{i}(v)$ are defined for $i$ sufficiently large. 
Then, by the invariant property of $\mu$ under the parabolic scaling of the Ricci flow, one has, for $s<v\in [-10, 0]$
\begin{eqnarray}
\label{eq-monotone}
&  &\mu(g_i(v), Q_i(T-t_i) - v) - \mu(g_i(s), Q_i(T-t_i) - s)  \nonumber\\
&=& \mu(g(t_i + \frac{v}{Q_i}), T-t_i-\frac{v}{Q_i})
- \mu(g(t_i+\frac{s}{Q_i}),T-t_i-\frac{s}{Q_i})\nonumber\\& =& \int_{t_i+\frac{s}{Q_i}}^{t_i+\frac{v}{Q_i}} \frac{d}{dt}\mu(g(t),T-t){dt} \nonumber \\
&\ge&  \int_{t_i+\frac{s}{Q_i}}^{t_i+\frac{v}{Q_i}}{\int_{M}2\tau}(4\pi\tau)^{-n/2}\cdot |\ric + \nabla\nabla f - \frac{g}{2\tau}|^2 e^{-f}\, dvol_{g(t)} dt  \nonumber \\
&=& 2\int_s^v\int_{M}m_{i}(r)(4\pi m_{i}(r))^{-n/2}|\ric(g_i(r)) + \nabla\nabla f - \frac{g_i}{2 m_{i}(r)}|^2 e^{-f}  dvol_{g_{i}(r)} dr
\end{eqnarray}
where, for simplicity, we have denoted
\begin{equation}
 m_{i}(r) = Q_i(T-t_i) - r.
\end{equation}
Since we are assuming the flow develops a type I singularity at $T$, we have 
\begin{equation}
\label{eq-finite}
\lim_{i\to\infty} Q_i(T-t_i) = a < \infty.
\end{equation}
Thus, by (\ref{curgap}), one has for $r\in [-10, 0]$,
\begin{equation}
 \lim_{i\rightarrow\infty} m_{i}(r) = a- r> 0.
\label{gapr}
\end{equation}
By Lemma \ref{lem-mu} and by the monotonicity of $\mu(g(t),T-t)$ (see (\ref{mon-mu})), we have 
\begin{equation}
\label{eq-mu-bound}
\mu(g(0), T) \le \mu(g(t),T-t) \le 0.
\end{equation}
Estimate (\ref{eq-mu-bound}) implies that there exists a finite $\lim_{t\to T} \mu(g(t),T-t)$  which has as a consequence that the 
left hand side of (\ref{eq-monotone}) tending to zero as $i\to\infty$. Letting $i\to \infty$ in (\ref{eq-monotone}) and using (\ref{gapr}), we get
\begin{equation}
\label{eq-soliton}
\lim_{i\to\infty} \int_s^{v}\int_M { (a-r)[4\pi (a-r)]^{-n/2}}|Ric(g_i) + \nabla\nabla f - \frac{g_i}{2(a-r)}|^2 e^{-f}\, dvol_{g_{i}(r)}\, dr = 0.
\end{equation}

We would like to say that we can extract a subsequence so that $f(\cdot,t_i + \frac{r}{Q_{i}})$  converges smoothly to a smooth 
function $f_{\infty}(r)$ on $(M_{\infty}, g_{\infty}(r))$, which will then be a potential function for a limiting gradient shrinking Ricci soliton $g_{\infty}$. In order
 to do that, we need to get some uniform estimates for $f(\cdot,t_i + \frac{r}{Q_{i}} )$. 
{The equation satisfied by $f(t_{i} + \frac{r}{Q_{i}})$ in (\ref{eq-monotone}) is
\begin{equation}
\label{eq-f10}
 (T-t_{i} -\frac{r}{Q_{i}}) (2\Delta f -\abs{\nabla f}^2 + R) + f-n = \mu (g(t_{i} + \frac{r}{Q_{i}}), T- t_{i} - \frac{r}{Q_{i}}).
\end{equation}
Let $f_{i}(\cdot, r) = f(\cdot, t_{i} + \frac{r}{Q_{i}}).$ Then
\begin{equation*}
 [Q_{i}(T-t_{i})-r] (2 \Delta_{g_{i}(r)} f_i(r) - \abs{\nabla_{g_{i}(r)} f_i(r)} + R(g_i(r))) + f_i(r)-n = \mu (g_{i}(r), Q_{i}(T-t_{i})-r).
\end{equation*}
 Define $\phi_i(\cdot, r) = e^{-f_i(\cdot, r)/2}$. This function $\phi_{i}(\cdot, r)$ satisfies a nice elliptic equation
\begin{equation}
 [Q_{i}(T-t_{i})-r] (-4\Delta_{g_{i}(r)} + R(g_i(r)))\phi_i  = 2\phi_i\log \phi_i + (\mu (g_{i}(r), Q_{i}(T-t_{i})-r) + n)\phi_i.
\label{eq-min}
\end{equation}
Recall that, in this Theorem \ref{prop-typeI}, we consider $r\in [-10, 0]$. We also take the freedom to suppress certain dependences on $r$ whenever no possible confusion may arise.\\
\h Our first estimates are uniform global $W^{1,2}$ estimates for $\phi_{i}(r)$ as shown in the following
\begin{lemma}
\label{lem-W12}
There exists a uniform constant $C$ so that for all $r\in [-10, 0]$ and all $i$, one has
$$\int_M \phi_i^2(\cdot,r)\, dvol_{g_i(r)} + \int_M |\nabla_{g_i(r)}\phi_i(\cdot,r)|^2\, dvol_{g_i(r)} \le C(Q_i(T-t_i) - r)^{n/2} \le \tilde{C}.$$
\end{lemma}

\begin{proof}
Note that $\phi_i(r)$ satisfies the $L^{2}$-constraint
\begin{equation*}
 \int_{M}[4\pi m_{i}(r)]^{-n/2} (\phi_i(r))^2 dvol_{g_i(r)} =1
\end{equation*}
and is in fact smooth \cite{Rothaus}. Here, we have used $m_{i} (r) = Q_{i}(T-t_{i})-r$. \\

To simplify, let $F_{i}(r) = \frac{\phi_{i}(r)}{c_{i}(r)}$ where $c_{i}(r) = [4\pi m_{i}(r)]^{n/4} $. Then
\begin{equation*}
 \int_{M} (F_{i}(r))^2 dvol_{g_i(r)} =1
\end{equation*}
and the equation for $F_i(r)$ becomes
\begin{equation*}
 m_{i}(r) (-4\Delta_{g_{i}(r)} + R(g_i(r)))F_{i}(r)  = 2 F_{i}(r)\log F_{i}(r) + (\mu (g_{i}(r), m_{i}(r)) + n + 2 \log c_{i}(r))F_{i}(r).
\end{equation*}
Introduce
\begin{equation*}
 \mu_{i}(r) = \mu (g_{i}(r), m_{i}(r)) + n + 2 \log c_{i}(r).
\end{equation*}
Then
\begin{equation*}
 -\Delta_{g_{i}(r)} F_{i} = \frac{1}{2 m_{i}(r)} F_{i}\log F_{i} + (\frac{\mu_{i}(r)}{4 m_{i}(r)} - \frac{1}{4}R(g_{i}(r))) F_{i}.
\end{equation*}
Multiplying the above equation by $F_{i}(r)$ and integrating over $M$, we get
\begin{multline}
 \int_{M} \abs{\nabla_{g_{i}} F_{i}}^2 dvol_{g_i(r)} = \frac{1}{2 m_{i}(r)} \int_{M}F^2_{i}\log F_{i} dvol_{g_i(r)}\\ + 
\int_{M}(\frac{\mu_{i}(r)}{4 m_{i}(r)} - \frac{1}{4}R(g_{i})) F^2_{i} dvol_{g_i(r)}.
\label{boundlog}
\end{multline}
Because $\int_{M} (F_{i}(r))^2 dvol_{g_i(r)} =1$, by Jensen's inequality for the logarithm,
\begin{eqnarray}
\int_{M}F^2_{i}\log F_{i} dvol_{g_i(r)} &=& \frac{n-2}{4}\int_{M}F^2_{i}\log F^{\frac{4}{n-2}}_{i} dvol_{g_i(r)}
\nonumber\\ &\leq& \frac{n-2}{4}\log \int_{M}F^{2 + \frac{4}{n-2}}_{i}dvol_{g_i(r)}
= \frac{n-2}{4}\log \int_{M}F^{\frac{2n}{n-2}}_{i}dvol_{g_i(r)}.
 \label{Jensen}
\end{eqnarray}
On the other hand, we recall the following Sobolev inequality due to Hebey-Vaugon \cite{HV} (see also Theorem 5.6 in Hebey \cite{Hebey})
\begin{theorem}[Hebey-Vaugon]
\label{thm-hebey}
 For any smooth, compact Riemannian n-manifold $(M, g)$, $n\geq 3$ such that
\begin{equation*}
 \abs{Rm(g)}\leq \Lambda_{1}, \abs{\nabla_{g} Rm(g)}\leq \Lambda_{2}, inj_{(M, g)}\geq \gamma
\end{equation*}
one has a uniform constant $B(n, \Lambda_{1}, \Lambda_{2},\gamma)$ so that, for any $u\in W^{1,2} (M)$
\begin{equation}
\left(\int_{M} \abs{u}^{\frac{2n}{n-2}} dvol_{g}\right)^{\frac{n-2}{n}} \leq C(n)\int_{M}\abs{\nabla u}^2 dvol_{g} + 
B(n, \Lambda_{1}, \Lambda_{2},\gamma) \int_{M} u^2 dvol_{g}.
\label{HVI}
\end{equation}
\label{HVT}
\end{theorem}
By Perelman's noncollapsing result, Theorem \ref{HVT} applies to $(M, g_{i}(r))$ with uniform constants $\Lambda_{1},\Lambda_{2},\gamma$, independent of $r\in [-10, 0]$ and $i$. 
In particular, letting $u= 
F_{i}(r)$
in (\ref{HVI}), we find that
\begin{equation}
\int_{M} (F_{i}(r))^{\frac{2n}{n-2}} dvol_{g_i(r)} \leq C(n) \left(\int_{M}\abs{\nabla_{g_{i}(r)} F_{i}(r)}^2 dvol_{g_i(r)}\right)^{\frac{n}{n-2}} + B(n, \Lambda_{1}, \Lambda_{2},\gamma).
\label{HVF}
\end{equation}
Combining (\ref{boundlog}), (\ref{Jensen}) and (\ref{HVF}), we obtain
\begin{multline}
 \int_{M} \abs{\nabla_{g_{i}} F_{i}}^2 dvol_{g_i(r)} \leq \frac{n-2}{8 m_{i}(r) }\log \int_{M}F^{\frac{2n}{n-2}}_{i}dvol_{g_i(r)}  + 
\int_{M}(\frac{\mu_{i}(r)}{4 m_{i}(r)} - \frac{1}{4}R(g_{i})) F^2_{i} dvol_{g_i(r)}\\ \leq   \frac{n-2}{8 m_{i}(r) }\log \left( C(n) \left(\int_{M}\abs{\nabla F_{i}}^2 
dvol_{g_i(r)}\right)^{\frac{n}{n-2}} + 
B(n, \Lambda_{1}, \Lambda_{2},\gamma) \right) \\ +
\int_{M}(\frac{\mu_{i}(r)}{4 m_{i}(r)} - \frac{1}{4}R(g_{i})) F^2_{i} dvol_{g_i(r)}.
\label{global1}
\end{multline}
Recall that $R(g_{i}(r))$ is uniformly bounded by our scaling and furthermore
\begin{equation*}
 \lim_{i\rightarrow\infty} Q_{i} (T-t_{i}) = a \in [\frac{1}{8}, \infty).
\end{equation*}
Thus, if $r\in [-10, 0]$, then (\ref{global1}) gives 
a global uniform bound for $ \int_{M} \abs{\nabla_{g_{i}(r)} F_{i}(r)}^2 dvol_{g_i(r)}$. Hence, we have a global uniform bound for 
$ \int_{M} \abs{\nabla_{g_{i}(r)} \phi_{i}(r)}^2dvol_{g_i(r)}$
because $\phi_{i}(r) = c_{i}(r)F_{i}(r)$.
\end{proof}
Now, elliptic $L^p$ theory gives uniform $C^{1,\alpha}$ estimates for $\phi_i(r)$ on compact sets \cite{GT}. To get higher order derivative estimates on $\phi_i(r)$, in 
order  to 
be able to conclude that for a suitably chosen sequence of points $q_i$ around which we decide to take the limit we have $f_{\infty}(r) = -2\log \phi_{\infty}(r)$, for a 
smooth function $f_{\infty}(r)$ (where $f_{\infty}(r)$ is the limit of $f_i(r)$ and $\phi_{\infty}(r)$ is the limit of $\phi_i(r)$), it is crucial to prove 
that $\{\phi_i(r)\}$ stay uniformly bounded from below on compact sets around $q_i$. 

\h In (\ref{eq-soliton}), take $s = -10$ and $v = 0$. For each $i$, let $r_i\in [-10,0]$ be such that
\begin{eqnarray*}
& &(a-r_i)[4\pi(a-r_i)]^{-n/2}\abs{Ric(g_i(r_i)) + \nabla\nabla f(t_i+\frac{r_i}{Q_i}) - \frac{g_i}{2(a-r_i)}}^2 e^{-f(t_i+\frac{r_i}{Q_i})}\, dvol_{g_i(r_i)} \\
&\le&  (a-r)[4\pi(a-r)]^{-n/2}\abs{Ric(g_i(r)) + \nabla\nabla f(t_i+\frac{r}{Q_i}) - \frac{g_i}{2(a-r)}}^2 e^{-f(t_i+\frac{r}{Q_i})}\, dvol_{g_i(r)}, 
\end{eqnarray*}
for all $r\in [-10,0]$.
Take $q_i\in M$ at which the maximum of $\phi_i(r_i)$ over $M$ has been achieved and denote also by $(M_{\infty}, g_{\infty}(t), q)$ the smooth pointed Cheeger-Gromov 
limit of the rescaled  sequence of metrics $(M, g_i(t), q_{i})$, defined as above.  
Take any compact set $K\subset M_{\infty}$ containing $q$. Let $\psi_i : K_i \to K$ be the diffeomorphisms from the definition of Cheeger-Gromov 
convergence of $(M, g_i, q_i)$ to $(M_{\infty}, g_{\infty}, q)$ and $K_i\subset M$.
Following the previous notation, consider the functions $F_i(r_i)$, $\phi_{i}(r_{i})$ and denote them for simplicity by $F_i$ and $\phi_{i}$, respectively. 
We will also denote the metric $g_{i}(r_i)$ shortly by $g_{i}$.
\begin{lemma}
\label{lem-C1alpha}
 For any $\alpha\in (0,1)$, there is a uniform constant $C(\alpha)$ so that
\begin{equation}
\label{eq-C1alpha}
 \norm{F_{i}}_{C^{1,\alpha}(M)}\leq C(\alpha).
\end{equation}
\end{lemma}
\begin{proof}The proof is via boostrapping and rather standard for the equation satisfied by $F_{i}$
\begin{equation}
 -\Delta_{g_{i}} F_{i} = \frac{1}{2 m_{i}(r)} F_{i}\log F_{i} + (\frac{\mu_{i}(r)}{4 m_{i}(r)} - \frac{1}{4}R(g_{i})) F_{i}.
\label{Feq}
\end{equation}
The reason that bootstrapping works is simple. If $F_{i}$ is uniformly
bounded in $L^{p}(K_{i})$, where $K_i\in M$ is a compact set, then $F_{i}\log F_{i}$ is uniformly bounded in $L^{p-\delta}(K_{i})$ for any $\delta>0$. The standard local parabolic estimates will give us (\ref{eq-C1alpha}) which will be independent of a compact set since we have uniform global $W^{1,2}$ bound on $F_i$.
\end{proof}
\h  Let us now discuss how to get higher order derivatives estimates for $F_{i}$. Covariantly differentiating (\ref{Feq}), 
commuting derivatives, and noting that
\begin{equation*}
 -\Delta_{g_{i}}\partial_{l} F_{i} = - \partial_{l}\Delta_{g_{i}} F_{i} - \ric(g_{i})_{lk} g^{kp}_{i}\partial_{p}F_{i}
\end{equation*}
we get
\begin{multline}
 -\Delta_{g_{i}}\partial_{l} F_{i} = \frac{1}{2 m_{i}(r)} \partial_{l} F_{i}\log F_{i} + (
\frac{2 + \mu_{i}(r)}{4 m_{i}(r)}- \frac{1}{4}R(g_{i})) \partial_{l}F_{i} \\- \frac{1}{4}\partial_{l}R(g_{i}) F_{i}
- \ric(g_{i})_{lk} g^{kp}_{i}\partial_{p}F_{i}.
\label{Fdeq}
\end{multline}
The major obstacle in applying $L^{p}$ theory to get uniform $C^{1,\alpha}$ estimates for $ \partial_{l}F_{i}$ is the term $\partial_{l} F_{i}\log F_{i}$.
This emanates from the potential smallness of $\abs{F_{i}}$, though we have already had a nice uniform upper bound on it. Thus, to proceed further, we need to bound
$\abs{F_{i}}$ uniformly from below. Equivalently, we will prove in Lemma \ref{lem-lower-phi} that $\phi_{i}$
stays uniformly bounded from below on $K_{i}$. \\
\h As the first step, we bound $\phi_{i}(q_{i})$ from below. This is simple. If we apply the maximum principle to (\ref{eq-f10}) we obtain 
$\min_{M} f_i \le C$, where $f_i = f_i(r_i)$, for a uniform constant $C$. This can be seen as follows. Define $\alpha_{i} = Q_{i}(T-t_{i})$. At the minimum of $f_{i}$, we have
\begin{equation*}
\frac{f_{i}-n}{\alpha_{i}-r_i} =\frac{\mu(g_{i}(r_i), \alpha_{i}-r_i)}{\alpha_{i}-r_i} - R (g_{i}(r_i)) -2\Delta_{g_{i}(r_i)} f_{i}\leq \frac{\mu (g_{i}(r_i), \alpha_{i}-r_i)}
{\alpha_{i}-r_i} -R (g_{i}(r_i)).
\end{equation*}
Thus,
\begin{multline}
\label{eq-need-sc}
f_{i}\leq n + \mu(g_{i}(r_i), \alpha_{i}-r_i) -R(g_{i}(r_i))(\alpha_{i}-r_i)\\ \leq n + \mu(g_{i}(r_i), \alpha_{i}-r_i) + \frac{C}{Q_{i}}[Q_{i}(T-t_{i}) -r_i] \leq C,
\end{multline}
where we have used the fact that $R(\cdot,t) \ge -C$ on $M$, for all $t\in [0,T)$ (see (\ref{lower-scar})). This 
implies $\phi_i(q_i) \ge \delta > 0$ for all $i$,  with a uniform constant $\delta$. 

Let $K\subset M_{\infty}$ and $K_i\subset M$ be compact sets as before. Also recall that $m_{i}(r_{i}) = Q_{i}(T-t_{i}) -r_{i}$.

\begin{lemma}
\label{lem-lower-phi}
For every compact set $K\subset M_{\infty}$ there exists a uniform constant $C(K)$ so that 
$$\phi_i\ge C(K), \,\,\, \mbox{on} \,\,\, K_i, \,\,\ \mbox{for all} \,\,\, i.$$
\end{lemma}

\begin{proof}
Assume the lemma is not true and that there exist points $P_i\in K_i$ so that $\phi_i(P_i) \le 1/i \to 0$ as $i\to\infty$. Assume 
$\psi_{i}(P_i)$ converge to a point $P\in K$. Then $\phi_{\infty}(P) = 0$. Take a smooth function $\eta\in C_0^{\infty}(M_{\infty})$, compactly supported in 
$K\backslash \{P\}$. Then $\psi_{i}^*\eta \in C_0^{\infty}(M)$, compactly supported in $K_i\backslash\{P_i\}$. Multiplying (\ref{eq-min}) by $\psi_{i}^*\eta$, assuming $\lim_{i\to\infty} r_i = r_0$, and 
then integrating by parts, we get
$$\int_M m_{i}(r_{i})\cdot(4\nabla\phi_i\nabla (\psi_{i}^*\eta) + R_i\phi_i \psi_{i}^*\eta) - 2\phi_i \psi_{i}^*\eta\ln\phi_i - n\phi_i \psi_{i}^*\eta - 
\mu(g_i, m_{i}(r_{i}))\phi_i \psi_{i}^*\eta)\, dvol_{g_{i}(r_{i})} = 0.$$
Letting $i\to\infty$, using that $\phi_i\stackrel{C^{1,\alpha}}{\to} \phi_{\infty}$ locally, $\psi_i^*\eta \to \eta$ smoothly 
as $i\to \infty$, $\lim_{i\to\infty} R(g_i) = 0$, and $a-r_{0} :=\lim_{i\rightarrow \infty}m_{i}(r_{i}) \equiv 
\lim_{i\to\infty}\left(Q_i(T-t_i) -r_{i}\right)< \infty$, one finds that
$$\int_{M_{\infty}}(4(a - r_0) \nabla\phi_{\infty}\nabla \eta   - 2\eta\phi_{\infty} \ln\phi_{\infty} - n\phi_{\infty} \eta - \mu(g_{\infty}, a - r_0)
 \eta \phi_{\infty})\, dvol_{g_{\infty} (r_{0})} = 0.$$
Proceeding in the same manner as in Rothaus \cite{Rothaus} we can get that $\phi_{\infty} \equiv 0$ in some small ball around $P$. Using 
the connectedness argument, $\phi_{\infty} \equiv 0$ everywhere in $M_{\infty}$. That contradicts $\phi_{\infty}(q) \ge \delta > 0$.
\end{proof}

Having Lemma \ref{lem-lower-phi} and $C^{1,\alpha}$ uniform estimates on $\phi_i$, we see that the right hand side of (\ref{Fdeq}) is uniformly
bounded in $L^{2}(K_{i})$. Because $\log F_{i}$ is uniformly bounded on $K_{i}$, we can bootstrap (\ref{Fdeq}) to obtain $C^{1,\alpha}$ estimates for 
$\abs{\nabla _{g_{i}} F_{i}}$. Hence, one has uniform $C^{2,\alpha}$ estimates for $F_{i}$ on $K_{i}$. In terms of $\phi_{i}$, one has that
\begin{equation}
\label{eq-phii100}
 \norm{\phi_{i}}_{C^{2,\alpha}(K_{i})}\leq C(K, \alpha) (Q_i(T-t_i) - r_i)^{n/4}.
\end{equation}
One can differentiate (\ref{Fdeq}) again and obtain all higher order derivative estimates on $\phi_i$ and therefore all higher order derivative estimates on 
$f_i = f_{i}(r_i) = -2\log\phi_i$. However, for our purpose, $C^{2,\alpha}$ estimates suffice.\\
\h Then, using (\ref{eq-soliton}), for $s = -10$ and $v = 0$,  
\begin{eqnarray*}
& &\lim_{i\to\infty}10(a - r_i)(4\pi (a - r_i))^{-n/2}\int_M |\ric(g_i(r_i)) + \nabla\nabla f_i - \frac{g_i(r_i)}{2(a-r_i)}|^2 e^{-f_i}\, dvol_{g_{i}(r_i)} \\
&\le& \lim_{i\to\infty} \int_{-10}^0 \int_M (a-r)[4\pi(a-r)]^{-n/2} |\ric(g_i) + \nabla\nabla f_i - \frac{g_i}{2(a-r)}|^2 e^{-f_i}\, dvol_{g_{i}(r)}\, dr = 0.
\end{eqnarray*}
By Lemma \ref{lem-lower-phi} and (\ref{eq-soliton}), applying Arzela-Ascoli theorem on $f_i$ results in}
$$\ric_{\infty} + \nabla\nabla f_{\infty} - \frac{g_{\infty}}{2(a-r_0)} = 0.$$
Since $\ric_{\infty} \equiv 0$, we get
$$g_{\infty} = 2(a-r_0)\nabla\nabla f_{\infty},$$
and therefore $M_{\infty}$ is isometric to a standard Euclidean space $\mathbb{R}^n$; see, e.g., Proposition 1.1 in \cite{Ni}. It is now easy to see that 
\begin{equation}
\label{eq-gaussian}
f_{\infty} = \frac{|x|^2}{4(a-r_{0})},
\end{equation}
that is, the limiting manifold $(\mathbb{R}^n, g_{\infty}, q_{\infty})$ is a Gaussian shrinker.
\end{proof}

In Corollary \ref{cor-vol} we claim that if our solution is Type I and if it has uniformly bounded scalar curvature, then the volume can not go to zero.

\begin{proof}[Proof of Corollary \ref{cor-vol}]
Inspecting the proof of Theorem \ref{prop-typeI} and noting that our solution $g(t)$ is of type I, we see that, if we rescale the metrics by $g_{j} =
Q_{j} g(T +\frac{t}{Q_{j}})$ then we also get a Gaussian shrinker.
Now, using Perelman's pseudolcality theorem (\cite{Pe}, Theorem 10.3) and Theorem \ref{prop-typeI} we get
$$|\rem|_{g(t)} \le \frac{Q_j}{(\e r)^2}, \,\,\, \mbox{in} \,\,\, B_{g(T-\frac{(\e r)^2}{Q_{j}})}(q_j,\frac{\e r}{\sqrt{Q_j}}),$$
for all $t\in [T-\frac{(\e r)^2}{Q_{j}}, T)$ and all $j \ge j_0$, for sufficiently large $j_0$. Here $r$ is arbitrary and $\e$ is a small
number in Perelman's pseudolocality theorem. This tells us all metrics $g(t)$, for $t \ge T-\frac{(\e r)^2}{Q_{j_{0}}}$ are uniformly equivalent to each other in 
a fixed ball $B_{g(T-\frac{(\e r)^2}{Q_{j_{0}}})}(q_{j_0}, \frac{\e r}{Q_{j_0}})$ and therefore,
$$\vol_{g(t)}(M) \ge \vol_{g(t)}  (B_{g(T-\frac{(\e r)^2}{Q_{j_{0}}})}(q_{j_0}, \frac{\e r}{Q_{j_0}})) \ge  C(j_0) > 0.$$
\end{proof}

\begin{proof}[Proof of  Theorem \ref{lem-typeII}]
We will use many estimates and arguments that we have developed in the proof of Theorem \ref{prop-typeI}. Assume the flow does develop a type II singularity at $T$. Then we can pick a sequence of 
times $t_i\to T$ and points $p_i \in M$ as in \cite{Ha} so that the rescaled sequence of solutions $(M, g_i(t) := Q_i g(t_i + t/Q_i), p_i)$,  converges in a 
pointed Cheeger-Gromov sense to a Ricci flat, nonflat, complete, eternal solution $(M_{\infty}, g_{\infty}(t),p_{\infty})$. Here 
$Q_i := \max_{M\times[0,t_i]}|\rem|(x,t) \to \infty$ as $i\to\infty$. The reasons for getting Ricci flat metric are the same as in the proof of Theorem \ref{prop-typeI}.
Define
\begin{equation*}\alpha_i := (T-t_i)Q_i.
 \end{equation*}
Since we are assuming type II singularity occurring at $T$, we may assume that for a chosen sequence $t_i$ we have $\lim_{i\to\infty}\alpha_i = \infty$.\\
\h By Lemma \ref{lem-mu}  and  the monotonicity of $\mu$ we have $|\mu(g(t), T-t)| \le C$ for all $t\in [0,T)$.  Let $f_i(\cdot, s)$ be a smooth minimizer realizing  
$$\mu(g(t_i +s/Q_i), T-t_i - s/Q_i) = \mu(g_i(s), \alpha_i - s) =\inf \mathcal{W}(g(t_{i} + \frac{s}{Q_{i}}), f, T- t_{i} - \frac{s}{Q_{i}} )$$ over the set of all smooth functions $f$ satisfying 
\begin{equation*}
 [4\pi (T-t_{i} - \frac{s}{Q_{i}})]^{-n/2}\int_{M} e^{-f} dvol_{g(t_{i} + \frac{s}{Q_{i}})} =1.
\end{equation*}
  Then $f_i = f_i(\cdot,s)$ satisfies
\begin{equation}
\label{eq-f}
2\Delta_{g_i(s)} f_i - |\nabla_{g_{i}(s)} f_i|^2 + R_i + \frac{f_i - n}{\alpha_i - s} = \frac{\mu (g_{i}(s), \alpha_{i}-s)}{\alpha_i - s}.
\end{equation}
In terms of $\phi_i(x,s) = e^{-f_i(x,s)/2}$ this is equivalent to
\begin{equation}
\label{eq-u}
-4\Delta_{g_i(s)} \phi_i(s) + R (g_{i}(s))\phi_i(s) = \frac{2\phi_i(s)\log\phi_i(s)}{\alpha_i - s} + \frac{(\mu (g_{i}(s), \alpha_{i}-s) + n)\phi_i(s)}{\alpha_i - s},
\end{equation}
with
\begin{equation}
\label{eq-constraint}
\int_M (\phi_i(s))^2\, dvol_{g_i(s)} = (4\pi(\alpha_i - s))^{n/2}.
\end{equation}
In what follows, we fix $s=0$.
Define $\tilde{\phi}_i(\cdot) := \frac{\phi_i(\cdot,0)}{\beta_i}$, where 
\begin{equation}\beta_i:= \max_{M} \left(\phi_i(x,0) + \abs{\nabla_{g_{i}(0)} \phi_{i}(x, 0)}\right).
\label{betaeq} 
\end{equation}
This choice of $\beta_{i}$ gives us uniform $C^{1}$ estimates for $\tilde{\phi_{i}}$ on $M$. Thus, we can apply $L^{p}$ theory to get 
uniform  $C^{1,\alpha }$ estimates 
for $\tilde{\phi}_i$ on compact sets around the points where the maxima in (\ref{betaeq}) are achieved. To be more precise, we proceed as follows.\\
\h Take $q_i\in M$ at which this maximum in (\ref{betaeq}) has been achieved and denote also by $(M_{\infty}, g_{\infty}(t), q)$ the smooth pointed Cheeger-Gromov limit of the rescaled 
 sequence of metrics $(M, g_i(t), q_{i})$, defined as above. Lemma \ref{lem-W12}, Theorem \ref{thm-hebey} and standard elliptic $L^{p}$ estimates applied to (\ref{eq-u})  yield  
the estimates on $\beta_i$ in terms of the $W^{1,2}$ norm of $\phi_i$, with respect to metric $g_i(0)$, that is, there exists a uniform constant $C$ so that for all $i$,
$\beta_i \le C\alpha_i^{n/4}$, which implies 
\begin{equation}
\label{eq-beta}
\log\beta_i \le C_2\log\alpha_i + C_2,
\end{equation}
for some uniform constants $C_1$ and $C_2$. This can be proved  the same way  we obtained (\ref{eq-phii100}) in Theorem \ref{prop-typeI}. After dividing (\ref{eq-u}) by $\beta_i$ we get
\begin{equation}
\label{eq-phii}
-4\Delta_{g_i(0)}\tilde{\phi}_i + R(g_{i}(0))\tilde{\phi_i} = 2\tilde{\phi}_i\cdot\frac{\log\tilde{\phi}_i + \log\beta_i}{\alpha_i } + 
\frac{(\mu (g(t_i), T-t_i)+ n)\tilde{\phi}_i}{\alpha_i }.
\end{equation}
Since $(M, g_{i}(t), q_{i})$ converges in the pointed Cheeger-Gromov sense to $(M_{\infty}, g_{\infty}(t), q)$, and $\norm{\tilde{\phi_{i}}}_{C^{1}(M, g_{i}(0))}$
is uniformly bounded, we can get uniform  $C^{1,\alpha }$ estimates 
for $\tilde{\phi}_i$ on compact sets around points $q_i$. By Arzela-Ascoli 
theorem $\tilde{\phi}_i$ converges uniformly in the $C^1$ norm on compact sets around points $q_i$ to a smooth function $\tilde{\phi}_{\infty}$. 
We will show below that $\tilde{\phi}_{\infty}(\cdot)$ is a positive constant.\\
\h Indeed, if we apply the maximum principle to (\ref{eq-f}), similarly as in the proof of Theorem \ref{prop-typeI}, we obtain $\min_{M} f_i(\cdot, 0) \le C$, for a uniform constant $C$.  This 
implies $\log \beta_i \ge -C_1$ for a uniform constant $C_1$. 
In particular, there is a uniform constant $\delta>0$ such that for all $i$, one has
\begin{equation}
\beta_{i}\geq \delta>0.
 \label{eq-del}
\end{equation}
This together with (\ref{eq-beta}) and the $\lim_{i\to\infty}\alpha_i = \infty$ implies 
\begin{equation}
\label{eq-prop}
\lim_{i\to\infty}\frac{\log\beta_i}{\alpha_i} = 0.
\end{equation}
If we multiply (\ref{eq-phii}) by any cut off function $\eta_i = \psi_{i}^*\eta$ (where $\eta$ is any cut off function on $M_{\infty}$ and $\psi_i$ is a sequence of 
diffeomorphisms from the definition of Cheeger Gromov convergence) and integrate by parts we get
\begin{multline*}4\int_M\nabla\tilde{\phi}_i\nabla \eta_i\, dvol_{g_{i}(0)} = -\int_M R(g_{i}(0))\tilde{\phi}_i \eta_i\, dvol_{g_{i}(0)} \\ + 
2\int_{M}\eta_i\tilde{\phi}_i\cdot\frac{\log\tilde{\phi}_i  + \log\beta_i}{\alpha_i }\, dvol_{g_{i}(0)} -  \frac{\mu(g(t_i), T-t_i) + n}{\alpha_i }\int_M \eta_i\tilde{\phi}_i\, 
dvol_{g_{i}(0)}.
\end{multline*} 
If we let $i\to\infty$ in the previous identity, using (\ref{eq-prop}), the 
$\lim_{i\to\infty}\alpha_i = \infty$, $R(g_i(0)) \to 0$ uniformly on compact sets, $\tilde{\phi}_i\stackrel{C^{1}}{\to} \tilde{\phi}_{\infty}$, and uniform bounds on $\mu(g(t), T-t)$ we obtain
$$\int_M \nabla\tilde{\phi}_{\infty}\nabla\eta\, dvol_{g_{\infty}(0)} = 0.$$
This means $\Delta\tilde{\phi}_{\infty} = 0$ in the distributional sense. By Weyl's theorem, $\tilde{\phi}_{\infty}$ is a harmonic function on $M_{\infty}$. Since
 $(M_{\infty}, g_{\infty}(0))$ is 
a complete, Ricci flat manifold and $\phi_{\infty} \ge 0$, by the theorem of Yau \cite{Yau}, $\tilde{\phi}_{\infty} = C_{\infty}$ is a 
constant function on $M_{\infty}$. {At the same time, from the definition of $\tilde{\phi_{i}}$, we get, for $x$ in compact sets around 
points $q_i$
\begin{equation}
1 = \lim_{i\rightarrow \infty}\left(\tilde{\phi_i}(x) + \abs{\nabla_{g_{i}(0)}\tilde{ \phi_{i}}(x)}\right) = \tilde{\phi}_{\infty}(x) + \abs{\nabla_{g_{\infty}(0)} 
\tilde{\phi}_{\infty}(x)}\\ \equiv C_{\infty} .
\end{equation}
This implies, in particular $C_{\infty}\equiv 1>0.$\\
 }

\end{proof}

\section{Integral bounds on scalar curvature}
\label{sec-cor}

In this section we will prove Theorem \ref{int-bound} and Theorem \ref{int-bound1}. Observe that Theorem \ref{thm-main} is a special case of Theorem \ref{int-bound} 
when $\alpha = \infty$ in the case we deal with type I singularities only. A crucial ingredient in our arguments is the following result.
\begin{theorem} (Enders-M\"{u}ller-Topping, Theorem 1.4 \cite{EMT})
Let $g(t)$ be the solution to a Type I Ricci flow (\ref{eq-RF}) on $[0, T)$ and suppose that the flow develops a Type I singularity at $T$. 
Then for every sequence $\lambda_j\to\infty$, the
rescaled Ricci flows $(M,g_j(t))$ defined on $[-\lambda_jT,0)$ by
$g_j(t):=\lambda_j g(T+\frac{t}{\lambda_j})$ subconverge, in the Cheeger-Gromov sense, to a
normalized nontrivial gradient shrinking soliton in canonical form on $(-\infty, 0)$.
\label{EMT-thm}
\end{theorem}

\begin{proof}[Proof of Theorem \ref{int-bound}]
The proof is by contradiction. Assume the flow develops a type I singularity at $p\in M$ at $T < \infty$. Consider any sequence $\lambda_{j}\rightarrow \infty$ and
define $g_{j}(t): =\lambda_{j} g(T+\frac{t}{\lambda_{j}})$ where $t\in [-\lambda_{j}T, 0)$.  Then,
by Theorem \ref{EMT-thm}, the rescaled Ricci flows $(M,g_j(t), p)$ defined on $[-\lambda_jT,0)$ subconverge, in the Cheeger-Gromov sense, to a
normalized nontrivial gradient shrinking soliton $(M_{\infty}, g_{\infty}(t), p_{\infty})$ in canonical form on $(-\infty, 0)$.
Under the condition (\ref{int-alpha}), one has
\begin{equation*}
 \int_{M} \abs{R(g_{j}(t))}^{\alpha}dvol_{g_{j}(t)} = \frac{1}{\lambda_{j}^{\alpha-n/2}}\int_{M}\abs{R(g(T + \frac{t}{\lambda_{j}}))}^{\alpha} dvol_{ g(T + \frac{t}{\lambda_{j}})}
\leq \frac{C_{\alpha}}{\lambda_{j}^{\alpha-n/2}}\rightarrow 0.
\end{equation*}
\label{scar-flat}
Thus our limiting solution $(M_{\infty}, g_{\infty}(t), p_{\infty})$ is scalar flat. 
Arguing as in the proof of Theorem \ref{thm-main}, we see that $M_{\infty}$ is isometric to a 
standard Euclidean space $\mathbb{R}^n$. However, this contradicts the nontriviality of $M_{\infty}$.
\end{proof}

\begin{proof}[Proof of Theorem \ref{int-bound1}]  By H\"{o}lder inequality, it suffices to consider the case $\alpha =\frac{n +2}{2}$. Then
our integral bound is invariant under the usual parabolic scaling of the Ricci flow. 

The proof is by contradiction. Assume the flow develops a type I singularity at $p\in M$ at $T < \infty$. Consider any sequence $\lambda_{j}\rightarrow \infty$ and
define $g_{j}(t): =\lambda_{j} g(T+\frac{t}{\lambda_{j}})$ where $t\in [-\lambda_{j}T, 0)$.  Then,
by Theorem \ref{EMT-thm}, the rescaled Ricci flows $(M,g_j(t), p)$ defined on $[-\lambda_jT,0)$ subconverge, in the Cheeger-Gromov sense, to a
normalized nontrivial gradient shrinking soliton $(M_{\infty}, g_{\infty}(t), p_{\infty})$ in canonical form on $(-\infty, 0)$. Observe that  
\begin{equation*}
 \int_{-1}^0\int_M |R(g_j(t)|^{\alpha}\, dvol_{g_j(t)}\, dt = \int_{T-\frac{1}{\lambda_j}}^{T}\int_M |R(g(s))|^{\alpha}\, dvol_{g(s)}\, ds
\end{equation*}
Since $\int_0^T\int_M |R(g(t))|^{\alpha}\, dvol_{g(t)}\, dt < \infty$,  letting $j\rightarrow\infty$, we obtain
\begin{equation}
\int_{-1}^0\int_{M_{\infty}} |R(g_{\infty}(t)|^{\alpha} \, dvol_{g_{\infty}(t)}\, dt \le \lim_{j\to\infty}
\int_{T-\frac{1}{\lambda_j}}^{T}\int_M |R(g(s))|^{\alpha}\, dvol_{g(s)}\, ds = 0,
\label{intbound_infty}
\end{equation}
which implies $R(g_{\infty}(t)) \equiv 0$ on $M_{\infty}$, for $t\in [-1,0]$.
Thus our limiting solution 
$(M_{\infty}, g_{\infty}(t))$ is scalar flat. Arguing as in the proof of Theorem \ref{thm-main}, we see that $M_{\infty}$ is isometric to a 
standard Euclidean space $\mathbb{R}^n$. However, this contradicts the nontriviality of $M_{\infty}$.
\end{proof}

\end{document}